\documentclass[12pt,english]{amsart}
\usepackage[T1]{fontenc}
\usepackage[latin9]{inputenc}
\usepackage{geometry}
\usepackage{amssymb}

\makeatletter

\newcommand{\noun}[1]{\textsc{#1}}

 \theoremstyle{plain}
\newtheorem{thm}{Theorem}[section]
  \theoremstyle{plain}
  \newtheorem{cor}[thm]{Corollary}
  \theoremstyle{plain}
  \newtheorem{prop}[thm]{Proposition}
  \theoremstyle{remark}
  \newtheorem*{acknowledgement*}{Acknowledgement}
  \theoremstyle{definition}
  \newtheorem{defn}[thm]{Definition}
  \theoremstyle{plain}
  \newtheorem{lem}[thm]{Lemma}


\makeatother

\begin{document}

\title{Power series solution of the modified KdV equation}

\author{Tu Nguyen}

\begin{abstract}
We use the method of Christ \cite{MR2333210} to prove local well-posedness
of a modified mKdV equation in $\mathcal{F}L^{s,p}$ spaces.
\end{abstract}
\maketitle

\section{Introduction}

The mKdV equation on the torus is \begin{equation}
\left\{ \begin{array}{cc}
\partial_{t}u+\partial_{x}^{3}u+u^{2}\partial_{x}u=0 & \mbox{ }\\
u(\cdot,0)=u_{0}\end{array}\right.\label{eq:mKdV}\end{equation}
where $u\in H^{s}(\mathbb{T})$ is a real-valued function of $(x,t)\in\mathbb{T}\times\mathbb{R}$.
If $u$ is a smooth solution of (\ref{eq:mKdV}) then $\left\Vert u(\cdot,t)\right\Vert _{L^{2}(\mathbb{T})}=\left\Vert u_{0}\right\Vert _{L^{2}(\mathbb{T})}$
for all $t$, therefore $\widetilde{u}(x,t)=u(x+\frac{1}{2\pi}\left\Vert u_{0}\right\Vert _{L^{2}(\mathbb{T})}^{2}t,t)$
is a solution of

\begin{equation}
\left\{ \begin{array}{cc}
\partial_{t}u+\partial_{x}^{3}u+\left(u^{2}-\frac{1}{2\pi}\int_{\mathbb{T}}u^{2}(x,t)dx\right)\partial_{x}u=0\\
u(\cdot,0)=u_{0}\end{array}\right.\label{eq:modified mKdV}\end{equation}
Thus, (\ref{eq:modified mKdV}) and (\ref{eq:mKdV}) are essentially
equivalent. Using Fourier restriction norm method, Bourgain \cite{MR1215780}
showed that (\ref{eq:modified mKdV}) is locally well-posed when $s\geq1/2$,
with uniformly continuous dependence on the initial data $u_{0}$.
In \cite{MR1466164}, he also showed that when $s<1/2$, the solution
map is not $C^{3}$. Takaoka and Tsutsumi \cite{MR2097834} proved
local-wellposedness of (\ref{eq:modified mKdV}) when $s>3/8$. For
(\ref{eq:mKdV}), Kappeler and Topalov \cite{MR2131061} used inverse
scattering method to show wellposedness when $s\geq0$ and Christ,
Colliander and Tao \cite{MR2018661} showed that uniformly continuous
dependence on the initial data does not hold when $s<1/2$. Thus,
there is a gap between known local well-posedness results and the
space $H^{-1/2}(\mathbb{T})$ suggested by the standard scaling argument.

Recently, Gr\"unrock and Vega \cite{MR0003} showed local well-posedness
of the mKdV equation on $\mathbb{R}$ with initial data in \[
\widehat{H_{s}^{r}}(\mathbb{R}):=\{f\in\mathcal{D}'(\mathbb{R}):\left\Vert f\right\Vert _{\widehat{H_{s}^{r}}}:=\left\Vert \left\langle \cdot\right\rangle ^{s}\hat{f}(\cdot)\right\Vert _{L^{r'}}<\infty\},\]
when $2\geq r>1$ and $s\geq\frac{1}{2}-\frac{1}{2r}$. (for $r>\frac{4}{3}$,
this was obtained by Gr\"unrock \cite{MR2096258}). This is an extension
of the result of Kenig, Ponce and Vega \cite{MR1211741} that local-wellposedness
holds in $H^{s}(\mathbb{R})$ when $s\geq1/4$. Furthermore, as $\widehat{H_{s}^{r}}$
scales like $H^{\sigma}$ with $\sigma=s+\frac{1}{2}-\frac{1}{r}$,
this result covers spaces that have scaling exponent $-\frac{1}{2}+$.

There is also a related recent work of Gr\"unrock and Herr \cite{MR0002}
on the derivative nonlinear Schr\"odinger equation on $\mathbb{T}$.
Both \cite{MR0003} and \cite{MR0002} used a version of Bourgain's
method.

In this paper, we will apply the new method of solution of Christ
\cite{MR2333210} to investigate the local well-posedness of (\ref{eq:modified mKdV})
with initial data in \[
\mathcal{F}L^{s,p}(\mathbb{T}):=\{f\in\mathcal{D}'(\mathbb{T}):\left\Vert f\right\Vert _{\mathcal{F}L^{s,p}}:=\left\Vert \left\langle \cdot\right\rangle ^{s}\hat{f}(\cdot)\right\Vert _{l^{p}}<\infty\}.\]
Let $B(0,R)$ be the ball of radius $R$ centered at $0$ in $\mathcal{F}L^{s,p}(\mathbb{T})$.
Our main result is the following.

\begin{thm}
Suppose $s\geq1/2$, $1\leq p\leq\infty$ and $p'(s+1/4)>1$. Let
$W$ be the solution map for smooth initial data of (\ref{eq:modified mKdV}).
Then for any $R>0$ there is $T>0$ such that, the solution map $W$
extends to a uniformly continuous map from $B(0,R)$ to $C([0,T],\mathcal{F}L^{s,p}(\mathbb{T}))$.
\end{thm}
We note that the $\mathcal{F}L^{s,p}(\mathbb{T})$ spaces that are
covered by Theorem 1.1 have scaling index $\frac{1}{4}+$. The restriction
$s\geq1/2$ is due to the presence of the derivative in the nonlinear
term, and is only used to bound the operator $S_{2}$ in section 3.
The same restriction on $s$ is also required in the work on the derivative
nonlinear Schr\"odinger equation on $\mathbb{T}$ by Gr\"unrock
and Herr \cite{MR0002}. We believe, however, that the range of $p$
in Theorem 1.1 is not sharp.

Concerning (\ref{eq:mKdV}), we have the following.

\begin{cor}
Suppose $s\geq1/2$, $1\leq p\leq\infty$ and $p'(s+1/4)>1$. Let
$\widetilde{W}$ be the solution map for smooth initial data of (\ref{eq:modified mKdV}).
Then for any $R>0$ there is $T>0$ such that for any $c>0$, the
solution map $\widetilde{W}$ extends to a uniformly continuous map
from $B(0,R)\cap\left\{ \varphi:\left\Vert \varphi\right\Vert _{L^{2}}=c\right\} \subset\mathcal{F}L^{s,p}(\mathbb{T})$
to $C([0,T],\mathcal{F}L^{s,p}(\mathbb{T}))$.
\end{cor}
As in \cite{MR2333210}, the solution map $W$ obtained in Theorem
1.1 gives a weak solution of (\ref{eq:modified mKdV}) in the following
sense. Let $T_{N}$ be defined by $T_{N}u=\left(\chi_{[-N,N]}\widehat{u}\right)^{\vee}$.
Let $\mathcal{N}u:=\left(u^{2}-\frac{1}{2\pi}\int_{\mathbb{T}}u^{2}(x,t)dx\right)\partial_{x}u$
be the limit in $C([0,T],\mathcal{D}'(\mathbb{T}))$ of $\mathcal{N}(T_{N}u)$
as $N\rightarrow\infty$, provided it exists. 

\begin{prop}
Let $s$ and $p$ be as in Theorem 1.1. Let $\varphi\in\mathcal{F}L^{s,p}$
and $u:=W\varphi\in C([0,T],\mathcal{F}L^{s,p})$. Then $\mathcal{N}u$
exists and $u$ satisfies (\ref{eq:modified mKdV}) in the sense of
distribution in $(0,T)\times\mathbb{T}$.
\end{prop}
To prove these results, we will formally expand the solution map into
a sum of multilinear operators. These multilinear operators are described
in the section 2. Then we will show that if $u(\cdot,0)\in\mathcal{F}L^{s,p}$
then the sum of these operators converges in $\mathcal{F}L^{s,p}$
for small time $t$, when $s$ and $p$ satisfy the conditions of
Theorem 1.1. Furthermore, this gives a weak solution of (\ref{eq:modified mKdV}),
justifying our formal derivation.

\begin{acknowledgement*}
I would like to thank my advisor Carlos Kenig for suggesting the topic
and helpful conversations. I would also like to thank Axel Gr\"unrock
and Sebastian Herr for valuable comments and suggestions.
\end{acknowledgement*}

\section{Multilinear operators}

We rewrite (\ref{eq:modified mKdV}) as a system of ordinary differential
equations of the spatial Fourier series of $u$ (see formula (1.9)
of \cite{MR2097834}, and also Lemma 8.16 of \cite{MR1215780} ):

\begin{eqnarray}
\frac{d\hat{u}(n,t)}{dt}-in^{3}\hat{u}(n,t) & = & -i\sum_{n_{1}+n_{2}+n_{3}=n}\hat{u}(n_{1},t)\hat{u}(n_{2},t)n_{3}\hat{u}(n_{3},t)\nonumber \\
 &  & +i\sum_{n_{1}}\hat{u}(n_{1},t)\hat{u}(-n_{1},t)n\hat{u}(n,t)\label{eq:first}\\
 & = & \frac{-in}{3}\sum_{n_{1}+n_{2}+n_{3}=n}^{*}\hat{u}(n_{1},t)\hat{u}(n_{2},t)\hat{u}(n_{3},t)\nonumber \\
 &  & +in\hat{u}(n,t)\hat{u}(-n,t)\hat{u}(n,t),\nonumber \end{eqnarray}
where the star means the sum is taken over the triples satisfying
$n_{j}\ne n$, $j=1,2,3$.

Let $a(n,t)=e^{in^{3}t}\hat{u}(n,t)$, then $a_{n}(t)$ satisfy \[
\frac{da(n,t)}{dt}=-\frac{in}{3}\sum_{n_{1}+n_{2}+n_{3}=n}^{*}e^{i\sigma(n_{1},n_{2},n_{3})t}a(n_{1},t)a(n_{2},t)a(n_{3},t)+ina(n,t)a(-n,t)a(n,t),\]
where \[
\sigma(n_{1},n_{2},n_{3})=(n_{1}+n_{2}+n_{3})^{3}-n_{1}^{3}-n_{2}^{3}-n_{3}^{3}=3(n_{1}+n_{2})(n_{2}+n_{3})(n_{3}+n_{1}).\]
Or, in integral form, \begin{eqnarray}
a(n,t) & = & a(n,0)-\frac{in}{3}\int_{0}^{t}\sum_{n_{1}+n_{2}+n_{3}=n}^{*}e^{i\sigma(n_{1},n_{2},n_{3})s}a(n_{1},s)a(n_{2},s)a(n_{3},s)ds\label{eq: a integral sum}\\
 &  & +in\int_{0}^{t}\left|a(n,s)\right|^{2}a(n,s)ds.\nonumber \end{eqnarray}
We note that the triples in the sum are precisely those with $\sigma(n_{1},n_{2},n_{3})\ne0$.
If, $a$ is sufficiently nice, say $a\in C([0,T],l^{1})$ (which is
the case if $u\in C([0,T],H^{s}(\mathbb{T}))$ for large $s$) then
we can exchange the order of the integration and summation to obtain\begin{eqnarray}
a(n,t) & = & a(n,0)-\frac{in}{3}\sum_{n_{1}+n_{2}+n_{3}=n}^{*}\int_{0}^{t}e^{i\sigma(n_{1},n_{2},n_{3})s}a(n_{1},s)a(n_{2},s)a(n_{3},s)ds\label{eq: a sum integral}\\
 &  & +in\int_{0}^{t}\left|a(n,s)\right|^{2}a(n,s)ds.\nonumber \end{eqnarray}
Replacing the $a(n_{j},s)$ in the right hand side by their equations
obtained using (\ref{eq: a sum integral}), we get \begin{eqnarray}
a(n,t) & = & a(n,0)-\frac{in}{3}\sum_{n_{1}+n_{2}+n_{3}=n}^{*}a(n_{1},0)a(n_{2},0)a(n_{3},0)\int_{0}^{t}e^{i\sigma(n_{1},n_{2},n_{3})s}ds\nonumber \\
 &  & +in\left|a(n,0)\right|^{2}a(n,0)\int_{0}^{t}ds+\textrm{ additional terms }\nonumber \\
 & = & a(n,0)-\frac{n}{3}\sum_{n_{1}+n_{2}+n_{3}=n}^{*}\frac{a(n_{1},0)a(n_{2},0)a(n_{3},0)}{\sigma(n_{1},n_{2},n_{3})}(e^{i\sigma(n_{1},n_{2},n_{3})t}-1)\nonumber \\
 &  & +int\left|a(n,0)\right|^{2}a(n,0)+\textrm{ additional terms }\label{eq:a sum integral first step}\end{eqnarray}
The additional terms are those which depends not only on $a(m,0)$.
An example of the additional terms is\begin{eqnarray*}
-\frac{nn_{3}}{9}\sum_{n_{1}+n_{2}+n_{3}=n}^{*}a(n_{1},0)a(n_{2},0)\sum_{m_{1}+m_{2}+m_{3}=n_{3}}^{*}\int_{0}^{t}e^{i\sigma(n_{1},n_{2},n_{3})s}\int_{0}^{s}e^{i\sigma(m_{1},m_{2},m_{3})s'}\times\\
a(m_{1},s')a(m_{2},s')a(m_{3},s')ds'ds\end{eqnarray*}
We refer to section 2 of \cite{MR2333210} for more detailed description
of these additional terms. Then we can again use (\ref{eq: a sum integral})
for each appearance of $a(m,\cdot)$ in the additional terms, and
obtain more complicated terms. Continuing this process indefinitely,
we get a formal expansion of $a(n,t)$ as a sum of multilinear operators
of $a(m,0)$. 

We will now describes these operators and then show that their sum
converges. Again, we refer to section 3 of \cite{MR2333210} for more
detailed explanations. Each of our multilinear operators will be associated
to a tree, which has the property that each of its node has either
zero or three children. We will only consider trees with this property.
If a node $v$ of $T$ has three children, they will be denoted by
$v_{1},v_{2},v_{3}$. We denote by $T^{0}$ the set of non-terminal
nodes of $T$, and $T^{\infty}$ the set of terminal nodes of $T$.
Clearly, if $\left|T\right|=3k+1$ then $\left|T^{0}\right|=k$ and
$\left|T^{\infty}\right|=2k+1$. 

\begin{defn}
Let $T$ be a tree. Then $\mathcal{J}(T)$ is the set of $j\in\mathbb{Z}^{T}$
such that if $v\in T^{0}$ then \[
j_{v}=j_{v_{1}}+j_{v_{2}}+j_{v_{3}},\]
and either $j_{v_{i}}\ne j_{v}$ for all $i$, or $j_{v_{1}}=-j_{v_{2}}=j_{v_{3}}=j_{v}$.

We will denote by $v(T)$ be the root of $T$ and $j(T)=j(v(T))$.
For $j\in\mathcal{J}(T)$ and $v\in T^{0}$,\[
\sigma(j,v):=\sigma(j(v_{1}),j(v_{2}),j(v_{3})).\]

\begin{defn}
$\mathcal{R}(T,t)=\{s\in\mathbb{R}_{+}^{T^{0}}:\textrm{ if }v<w\textrm{ then }0\leq s_{v}\leq s_{w}\leq t\}$. 
\end{defn}
\end{defn}
Using these definitions, we can rewrite (\ref{eq:a sum integral first step})
as

\begin{eqnarray*}
a(n,t) & = & a(n,0)+\sum_{\left|T\right|=4}\omega_{T}\sum_{j\in\mathcal{J}(T),j(T)=n}na(j(v_{1}),0)a(j(v_{2}),0)a(j(v_{3}),0)\int_{\mathcal{R}(T,t)}c(j,v,s)ds\\
 &  & +\mbox{additional terms}\end{eqnarray*}
here $c(j,v,s)=e^{i\sigma(j,v)s}$, and $\omega_{T}$ is a constant
with $\left|\omega_{T}\right|\leq1$. 

Continuing the replacement process will lead to

\begin{eqnarray*}
a(n,t) & = & a(n,0)+\sum_{\left|T\right|<3k+1}\omega_{T}\sum_{j\in\mathcal{J}(T),j(T)=n}\prod_{u\in T^{0}}j_{u}\prod_{v\in T^{\infty}}a(j_{v},0)\int_{\mathcal{R}(T,t)}c(j,s)ds\\
 &  & +\mbox{additional terms}\end{eqnarray*}
where \[
c(j,s)=\prod_{v\in T^{0}}c(j,v,s)\]

We will show that the series \[
a(n,0)+\sum_{T}\omega_{T}\sum_{j\in\mathcal{J}(T),j(T)=n}\prod_{u\in T^{0}}j_{u}\prod_{v\in T^{\infty}}a(j_{v},0)\int_{\mathcal{R}(T,t)}c(j,s)ds\]
converges in $l^{p}$ to a weak solution of (\ref{eq:modified mKdV}).

\section{$l^{p}$ convergence}

\begin{defn}
For a tree $T$, $j\in\mathcal{J}(T)$, let\[
I_{T}(t,j)=\int_{\mathcal{R}(T,t)}c(j,s)ds,\]
 and \[
S_{T}(t)(a_{v})_{v\in T^{\infty}}(n)=\omega_{T}\sum_{j\in\mathcal{J}(T),j(T)=n}\prod_{u\in T^{0}}j_{u}\prod_{v\in T^{\infty}}a_{v}(j_{v})I_{T}(t,j).\]

\end{defn}
We first give an estimate for $I_{T}(t,j)$ which allows us to bound
$S_{T}$.

\begin{lem}
For $0\leq t\leq1$, $\left|I_{T}(j,t)\right|\leq(Ct)^{\left|T^{0}\right|/2}\prod_{v\in T^{0}}\left\langle \sigma(j,v)\right\rangle ^{-1/2}.$
\end{lem}
\begin{proof}
Let $v_{0}$ be the root of $T$. For $v\in T^{0}$, define the level
of $v$, denoted $l(v)$, to be the length of the unique path connecting
$v_{0}$ and $v$. Let $O$ be the set of $v\in T^{0}$ for which
$l(v)$ is odd, and $E$ those $v$ for which $l(v)$ is even.

First we fix the variables $s_{v}$ with $v\in E$, and take the integration
in the variables $s_{v}$ with $v\in O$. For each $v\in O$, the
result of the integration is\[
\frac{1}{\sigma(j,v)}\left(e^{i\sigma(j,v)s_{\tilde{v}}}-e^{i\sigma(j,v)\max\{s_{v(1)},s_{v(2)},s_{v(3)}\}}\right)\]
if $\sigma(j,v)\ne0$, and \[
s_{\tilde{v}}-\max\{s_{v(1)},s_{v(2)},s_{v(3)}\}.\]
 if $\sigma(j,v)=0$. Here $\widetilde{v}$ is the parent of $v$.
Thus, we obtain the factor\[
\prod_{v\in O}\left\langle \sigma(j,v\right\rangle ^{-1}\]
and an integral in $s_{v}$, $v\in E$ where the integrand is bounded
by $2^{\left|O\right|}$. As the domain of integration in $s_{v}$
with $v\in E$ has measure less than $t^{\left|E\right|}$, we see
that \[
\left|I_{T}(j,t)\right|\leq2^{\left|T^{0}\right|}t^{\left|E\right|}\prod_{v\in O}\left\langle \sigma(j,v)\right\rangle ^{-1}.\]
By switching the role of $O$ and $E$, we get \[
\left|I_{T}(j,t)\right|\leq2^{\left|T^{0}\right|}t^{\left|O\right|}\prod_{v\in E}\left\langle \sigma(j,v)\right\rangle ^{-1}.\]
Combining these two estimates, we obtain the lemma.
\end{proof}
By the previous lemma, \[
\left|S_{T}(t)(a_{v})_{v\in T^{\infty}}(n)\right|\leq(Ct)^{\left|T^{0}\right|/2}\sum_{j\in\mathcal{J}(T):j(T)=n}\prod_{u\in T^{0}}\left\langle \sigma(j,u)\right\rangle ^{-1/2}\left|j_{u}\right|\prod_{v\in T^{\infty}}\left|a_{v}(j_{v})\right|.\]
Let \[
\widetilde{S}_{T}(a_{v})_{v\in T^{\infty}}(n)=\sum_{j\in\mathcal{J}(T):j(T)=n}\prod_{u\in T^{0}}\left\langle \sigma(j,u)\right\rangle ^{-1/2}\left|j_{u}\right|\prod_{v\in T^{\infty}}\left|a_{v}(j_{v})\right|,\]
and \[
\widetilde{S}(a_{1},a_{2},a_{3})(n)=\sum_{n_{1}+n_{2}+n_{3}=n}^{*}\left|n\right|\left\langle \sigma(n_{1},n_{2},n_{3})\right\rangle ^{-1/2}\prod_{i=1}^{3}\left|a_{i}(n_{i})\right|+n\left|\prod a_{i}(n)\right|.\]
It is clear that \[
\widetilde{S}_{T}(a_{v})_{v\in T^{\infty}}=\widetilde{S}(\widetilde{S}_{T_{1}}(a_{v})_{v\in T_{1}^{\infty}},\widetilde{S}_{T_{2}}(a_{v})_{v\in T_{2}^{\infty}},\widetilde{S}_{T_{3}}(a_{v})_{v\in T_{3}^{\infty}}).\]
where $T_{i}$ is the subtree of $T$ that contains all nodes $u$
such that $u\leq v(T)_{i}$ (recall that $v(T)$ is the root of $T$).
Hence, to bound $S_{T}$, it suffices to bound $\widetilde{S}$. For
this purpose, we will use the following simple lemma.

\begin{lem}
Let $S$ be the multilinear operator defined by \[
S(a_{1},a_{2},a_{3})(n)=\sum_{n_{1}+n_{2}+n_{3}=n}m(n_{1},n_{2},n_{3})\prod_{j=1}^{3}a_{j}(n_{j}),\]
Let $1\leq p\leq\infty$. Then for any pair of indices $i\ne j\in\{1,2,3\}$,
\[
\left\Vert S(a_{1},a_{2},a_{3})\right\Vert _{l^{p}}\leq\sup_{n}\left\Vert m(n_{1},n_{2},n_{3})\right\Vert _{l_{i,j}^{p'}}\prod_{k=1}^{3}\left\Vert a_{k}\right\Vert _{l^{p}}.\]

\end{lem}
\begin{proof}
By Holder inequality, for any $n$,

\[
\left|S(a_{1},a_{2},a_{3})(n)\right|\leq\left\Vert m(n_{1},n_{2},n_{3})\right\Vert _{l_{i,j}^{p'}}\left\Vert \prod_{k=1}^{3}a_{k}\right\Vert _{l_{i,j}^{p}}\leq\sup_{n}\left\Vert m(n_{1},n_{2},n_{3})\right\Vert _{l_{i,j}^{p'}}\left\Vert \prod_{k=1}^{3}a_{k}\right\Vert _{l_{i,j}^{p}}\]
 Taking $l^{p}$-norm in $n$ we obtain the lemma.
\end{proof}
To show that $\widetilde{S}$ is a bounded multilinear map on $l^{s,p}:=\{a:\left\langle \cdot\right\rangle ^{s}a\in l^{p}\}$,
we will show the boundedness of $S$ on $l^{p}$ where $S$ has kernel
\[
m(n_{1},n_{2},n_{3})=\frac{\left\langle n\right\rangle ^{s}\left|n\right|}{\left\langle \sigma(n_{1},n_{2},n_{3})\right\rangle ^{1/2}\prod_{k=1}^{3}\left\langle n_{k}\right\rangle ^{s}}\,\,\,\,\mbox{where }n=n_{1}+n_{2}+n_{3}.\]
We split $S$ into sum of two operators $S_{1}$ and $S_{2}$ where
$S_{1}$ has convolution kernel \[
m_{1}(n_{1},n_{2},n_{3})=\frac{\left\langle n\right\rangle ^{s}\left|n\right|}{\prod_{k=1}^{3}\left\langle n_{k}\right\rangle ^{s}\left\langle n-n_{k}\right\rangle ^{1/2}}\,\,\,\,\mbox{if }n=n_{1}+n_{2}+n_{3},\,\,\, n_{i}\ne n\]
and $S_{2}$ has kernel \[
m_{2}(n_{1},n_{2},n_{3})=n/\left\langle n\right\rangle ^{2s}\,\,\,\,\mbox{if }n_{1}=-n_{2}=n_{3}=n.\]
Clearly, for $S_{2}$ to be bounded, we need $s\geq1/2$. It remains
to bound $S_{1}$, for which we have the following.

\begin{prop}
$S_{1}$ is bounded from $l^{p}\times l^{p}\times l^{p}$ to $l^{p}$
when $s\geq1/4$ and $p'(s+\frac{1}{4})>1$.
\end{prop}
\begin{proof}
In the proof, all the sums are taken over the triples $(n_{1},n_{2},n_{3})$
that satisfy the additional property that $n_{i}\ne n$, for all $1\leq i\leq3$.
Clearly, we can assume $n>0$. Note that if say $\left|n_{1}\right|\geq5n$
then as $\left|n_{2}+n_{3}\right|=\left|n-n_{1}\right|\geq4n$, at
least one of $n_{2}$ and $n_{3}$ has absolute value bigger than
$2n$. Also, we cannot have $\left|n_{i}\right|\leq n/4$ for all
$i$. Thus, up to permutation, there are four cases.
\begin{enumerate}
\item $\left|n_{1}\right|,\left|n_{2}\right|,\left|n_{3}\right|\in[n/4,5n]$
\item $\left|n_{1}\right|,\left|n_{2}\right|\in[n/4,5n]$, $\left|n_{3}\right|\leq n/4$
\item $\left|n_{1}\right|\in[n/4,5n]$, $\left|n_{2}\right|,\left|n_{3}\right|\leq n/4$
\item $\left|n_{1}\right|,\left|n_{2}\right|\geq2n$
\end{enumerate}
By the previous lemma, it suffices to show that in each of these four
regions, for some $i\ne j$ the $l_{i,j}^{p'}$-norm of $m$ is bounded.

\textbf{Case 1.} As $3n=\sum(n-n_{i})$ for some index $i$, say $i=3$,
we must have $\left|n-n_{3}\right|\sim n$. Since we also have $\left|n_{1}\right|,\left|n_{2}\right|\gtrsim n$,
\[
\left|m(n_{1},n_{2},n_{3})\right|\lesssim\frac{\left\langle n\right\rangle ^{1/2-s}}{\left\langle n_{3}\right\rangle ^{s}\left|(n-n_{1})(n-n_{2})\right|^{1/2}}.\]
We will use the following inequality \[
\left|\frac{1}{n_{3}(n-n_{2})}\right|=\left|\frac{1}{n_{1}}\left(\frac{1}{n_{3}}-\frac{1}{n-n_{2}}\right)\right|\leq\frac{1}{\left|n_{1}\right|}\left(\frac{1}{\left|n_{3}\right|}+\frac{1}{\left|n-n_{2}\right|}\right).\]

\begin{enumerate}
\item If $1/4\leq s\leq1/2$: then $\left\langle n_{3}\right\rangle ^{p'(1/2-s)}\lesssim\left\langle n\right\rangle ^{p'(1/2-s)}$,
so \begin{eqnarray*}
\left\Vert m\right\Vert _{l_{1,2}^{p'}}^{p'} & \lesssim & \sum_{\left|n_{1}\right|\leq5n}\frac{\left\langle n\right\rangle ^{p'(1/2-s)}}{\left|n-n_{1}\right|^{p'/2}}\sum_{\left|n_{2}\right|\leq5n}\frac{\left\langle n_{3}\right\rangle ^{p'(1/2-s)}}{\left(\left\langle n_{3}\right\rangle \left|n-n_{2}\right|\right)^{p'/2}}\\
 & \lesssim & \sum_{\left|n_{1}\right|\leq5n}\frac{\left\langle n\right\rangle ^{p'(1/2-s)}}{\left|n-n_{1}\right|^{p'/2}}\sum_{\left|n_{2}\right|\leq5n}\frac{\left\langle n\right\rangle ^{p'(1/2-s)}}{\left|n_{1}\right|^{p'/2}}\left(\frac{1}{\left|n-n_{2}\right|^{p'/2}}+\frac{1}{\left|n-n_{1}-n_{2}\right|^{p'/2}}\right)\\
 & \lesssim & \sum_{\left|n_{1}\right|\leq5n}\frac{\left\langle n\right\rangle ^{p'(1-2s)}A_{n}}{\left|(n-n_{1})n_{1}\right|^{p'/2}}\\
 & \lesssim & \left\langle n\right\rangle ^{p'(1-2s)}A_{n}\sum_{\left|n_{1}\right|\leq5n}\left(\frac{1}{n}(\frac{1}{\left|n-n_{1}\right|}+\frac{1}{\left|n_{1}\right|})\right)^{p'/2}\\
 & \lesssim & \left\langle n\right\rangle ^{p'(1/2-2s)}A_{n}^{2}.\end{eqnarray*}
where $\sum_{0<j<5n}j^{-p'/2}=A_{n}.$ As \[
A_{n}\lesssim\left\{ \begin{array}{cc}
n^{1-p'/2} & \mbox{ if }p'<2\\
\log\left\langle n\right\rangle  & \mbox{ if }p'=2\\
1 & \mbox{ if }p'>2\end{array}\right.\]
we easily check that $\left\langle n\right\rangle ^{(1/2-2s)p'}A_{n}^{2}$
is bounded by a constant, under our hypothesis on $s$ and $p'$.
\item If $s>1/2$: then $\left\langle n-n_{2}\right\rangle ^{p'(s-1/2)}\lesssim\left\langle n\right\rangle ^{p'(s-1/2)}$,
so \begin{eqnarray*}
\left\Vert m\right\Vert _{l_{1,2}^{p'}}^{p'} & \lesssim & \sum_{\left|n_{1}\right|\leq5n}\frac{\left\langle n\right\rangle ^{p'(1/2-s)}}{\left|n-n_{1}\right|^{p'/2}}\sum_{\left|n_{2}\right|\leq5n}\frac{\left\langle n-n_{2}\right\rangle ^{p'(s-1/2)}}{\left(\left\langle n_{3}\right\rangle \left|n-n_{2}\right|\right)^{p's}}\\
 & \lesssim & \sum_{\left|n_{1}\right|\leq5n}\frac{\left\langle n\right\rangle ^{p'(1/2-s)}}{\left|n-n_{1}\right|^{p'/2}}\sum_{\left|n_{2}\right|\leq5n}\frac{\left\langle n\right\rangle ^{p'(s-1/2)}}{\left|n_{1}\right|^{p's}}\left(\frac{1}{\left|n-n_{2}\right|^{p's}}+\frac{1}{\left|n-n_{1}-n_{2}\right|^{p's}}\right)\\
 & \lesssim & \sum_{\left|n_{1}\right|\leq5n}\frac{B_{n}}{\left|n-n_{1}\right|^{p'/2}\left|n_{1}\right|^{p's}}\\
 & \lesssim & B_{n}\sum_{\left|n_{1}\right|\leq5n}\left|n-n_{1}\right|^{p'(s-1/2)}\left(\frac{1}{n}(\frac{1}{\left|n-n_{1}\right|}+\frac{1}{\left|n_{1}\right|})\right)^{p's}\\
 & \lesssim & \left\langle n\right\rangle ^{-p'/2}B_{n}^{2}.\end{eqnarray*}
where $B_{n}=\sum_{0<j<5n}j^{-p's}.$ As \[
B_{n}\lesssim\left\{ \begin{array}{cc}
n^{1-p's} & \mbox{ if }p's<1\\
\log\left\langle n\right\rangle  & \mbox{ if }p's=1\\
1 & \mbox{ if }p's>1\end{array}\right.\]
 we easily check that $\left\langle n\right\rangle ^{-p'/2}B_{n}^{2}$
is bounded by a constant, under our hypothesis on $s$ and $p'$.
\end{enumerate}
~

\textbf{Case 2} This case can be treated in exactly the same way as
the first case, except when $n_{3}=0$. In the region $n_{3}=0$,
\begin{eqnarray*}
\left\Vert m\right\Vert _{l_{1,3}^{p'}}^{p'} & \lesssim & \sum_{n_{1}}\frac{\left\langle n\right\rangle ^{p'(1/2-s)}}{\left|n_{1}(n-n_{1})\right|^{p'/2}}\leq\sum_{n_{1}}\left\langle n\right\rangle ^{-p's}\left(\frac{1}{\left|n_{1}\right|^{p'/2}}+\frac{1}{\left|n-n_{1}\right|^{p'/2}}\right)\\
 & \lesssim & \left\langle n\right\rangle ^{-p's}A_{n}\lesssim1\end{eqnarray*}

\textbf{Case 3} As $\left|n_{1}\right|,\left|n-n_{2}\right|,\left|n-n_{3}\right|\sim n$,
\[
\left|m(n_{1},n_{2},n_{3})\right|\lesssim\frac{1}{\left\langle n_{2}\right\rangle ^{s}\left\langle n_{3}\right\rangle ^{s}\left|n_{2}+n_{3}\right|^{1/2}}.\]

Without loss of generality, we can suppose $\left|n_{3}\right|\geq\left|n_{2}\right|$
\begin{enumerate}
\item If $\left|n_{2}\right|<\left|n_{3}\right|/2$:\begin{eqnarray*}
\left\Vert m\right\Vert _{l_{2,3}^{p'}}^{p'} & \lesssim & \sum_{0\leq\left|n_{2}\right|\leq n/4}\frac{1}{\left\langle n_{2}\right\rangle ^{p's}}\sum_{n/4\geq\left|n_{3}\right|>2n_{2}}\frac{1}{\left\langle n_{3}\right\rangle ^{p'(s+1/2)}}\\
 & \lesssim & \sum_{0\leq\left|n_{2}\right|\leq n/4}\frac{1}{\left\langle n_{2}\right\rangle ^{p'(2s+1/2)-1}}\\
 & \lesssim & 1\end{eqnarray*}
if $(s+1/4)p'>1$.
\item If $\left|n_{2}\right|\geq\left|n_{3}\right|/2$:\begin{eqnarray*}
\left\Vert m\right\Vert _{l_{2,3}^{p'}}^{p'} & \lesssim & \sum_{\left|n_{3}\right|\leq n/4}\frac{1}{\left\langle n_{3}\right\rangle ^{2p's}}\sum_{\left|n_{3}\right|\geq n_{2}\geq\left|n_{3}\right|/2}\frac{1}{\left\langle n_{3}+n_{2}\right\rangle ^{p'/2}}\\
 & \lesssim & \sum_{\left|n_{3}\right|\leq n/4}\frac{1}{\left\langle n_{3}\right\rangle ^{2p's}}\max\{\log\left\langle n_{3}\right\rangle ,\left\langle n_{3}\right\rangle ^{-p'/2+1}\}\\
 & \lesssim & \sum_{\left|n_{3}\right|\leq n/4}\frac{\log\left\langle n_{3}\right\rangle }{\left\langle n_{3}\right\rangle ^{2p's}}+\sum_{\left|n_{3}\right|\leq n/4}\frac{1}{\left\langle n_{3}\right\rangle ^{p'(2s+1/2)-1}}\lesssim1\end{eqnarray*}
as $2p's\geq p'(s+1/4)>1$.
\end{enumerate}
\textbf{Case 4} $\left|n_{1}\right|,\left|n_{2}\right|>2n$: Note
that in this case, $\left|n_{1}\right|\sim\left|n-n_{1}\right|$ and
$\left|n_{2}\right|\sim\left|n-n_{3}\right|$.
\begin{enumerate}
\item If $\left|n_{3}\right|,\left|n-n_{3}\right|\geq n/2:$ we have \[
\left|m(n_{1},n_{2},n_{3})\right|\lesssim\frac{\left\langle n\right\rangle ^{1/2}}{\left\langle n_{1}\right\rangle ^{s+1/2}\left\langle n_{2}\right\rangle ^{s+1/2}},\]
hence\begin{eqnarray*}
\left\Vert m\right\Vert _{l_{1,2}^{p'}}^{p'} & \lesssim & \left\langle n\right\rangle ^{p'/2}\sum_{\left|n_{1}\right|,\left|n_{2}\right|>2n}\frac{1}{\left\langle n_{1}\right\rangle ^{p'(s+1/2)}\left\langle n_{2}\right\rangle ^{p'(s+1/2)}}\\
 & \lesssim & \frac{\left\langle n\right\rangle ^{p'/2}}{\left\langle 2n\right\rangle ^{p'(2s+1)-2}}\lesssim1.\end{eqnarray*}

\item If $\left|n_{3}\right|<n/2$: then $\left|n_{1}\right|\sim\left|n_{2}\right|$
and $\left|n-n_{3}\right|\geq n/2$, so \[
\left|m(n_{1},n_{2},n_{3})\right|\lesssim\frac{n^{s+1/2}}{\left\langle n_{1}\right\rangle ^{2s+1}\left\langle n_{3}\right\rangle ^{s}},\]
hence\[
\left\Vert m\right\Vert _{l_{1,3}^{p'}}^{p'}\lesssim B_{n}\sum_{\left|n_{1}\right|>2n}\frac{n^{p'(s+1/2)}}{\left\langle n_{1}\right\rangle ^{p'(2s+1)}}\lesssim\frac{B_{n}}{n^{p'(s+1/2)-1}}\lesssim1\]

\item If $\left|n-n_{3}\right|<n/2$: then $\left|n_{1}\right|\sim\left|n_{2}\right|$
and $\left|n_{3}\right|\sim n$. Hence, \[
\left|m(n_{1},n_{2},n_{3})\right|\lesssim\frac{n}{\left\langle n_{1}\right\rangle ^{2s+1}\left\langle n-n_{3}\right\rangle ^{1/2}}.\]
Therefore, \begin{eqnarray*}
\left\Vert m\right\Vert _{l_{1,3}^{p'}}^{p'} & \lesssim & \sum_{\left|n_{1}\right|\geq2n}\sum_{n/2<n_{3}<3n/2}\frac{n^{p'}}{\left\langle n_{1}\right\rangle ^{p'(2s+1)}\left\langle n-n_{3}\right\rangle ^{p'/2}}\\
 & \lesssim & \sum_{\left|n_{1}\right|\geq2n}\frac{A_{n}n^{p'}}{\left\langle n_{1}\right\rangle ^{p'(2s+1)}}\lesssim\frac{A_{n}}{n^{2p's-1}}\lesssim1\end{eqnarray*}

\end{enumerate}
This concludes the proof of the proposition.
\end{proof}
~

\begin{proof}
[Proof of Theorem 1.1] Let $u_{0}\in\mathcal{F}L^{s,p}$ and $a(n)=\widehat{u_{0}}(n)$.
By the previous proposition, \[
\left\Vert S_{T}((a_{v})_{v\in T^{\infty}})\right\Vert _{l^{s,p}}\leq C^{\left|T^{0}\right|}t^{\left|T^{0}\right|/2}\prod_{v\in T^{\infty}}\left\Vert a_{v}\right\Vert _{l^{s,p}}.\]
 Hence, the sum \begin{eqnarray}
 &  & \left\Vert a(n,0)+\sum_{T}\sum_{j\in\mathcal{J}(T),j(T)=n}\prod_{u\in T^{0}}j_{u}\prod_{v\in T^{\infty}}a(j_{v},0)\int_{\mathcal{R}(T,t)}c(j,s)ds\right\Vert _{l^{s,p}}\leq\nonumber \\
 &  & \sum_{T}\left\Vert S_{T}(a,\ldots,a)\right\Vert _{l^{s,p}}\leq\sum_{k=0}^{\infty}(Ct)^{k/2}\left\Vert a\right\Vert _{l^{s,p}}^{2k+1}=\frac{\left\Vert u_{0}\right\Vert _{\mathcal{F}L^{s,p}}}{1-\sqrt{Ct}\left\Vert u_{0}\right\Vert _{\mathcal{F}L^{s,p}}^{2}}.\label{eq:converge}\end{eqnarray}
converges for all $t\lesssim\min\{1,\left\Vert u_{0}\right\Vert _{\mathcal{F}L^{s,p}}^{-4}\}$.
Let $a(n,t)$ denote this sum, and define the solution map $u=Wu_{0}$
by $\widehat{u}(n,t)=e^{-in^{3}t}a(n,t)$. It follows from (\ref{eq:converge})
that $W$ is uniformly continuous. It remains to show that $W$ extends
the solution maps for smooth initial data.

From the definition of $S_{T}$, it is clear that $a(n,t)$ satisfies
the equation (\ref{eq: a sum integral}). Let $u_{N}(0)=\left(\chi_{[-N,N]}\widehat{u_{0}}\right)^{\vee}$
and $u_{N}=W(u_{N}(0))$. As $\left\Vert u_{N}(\cdot,0)\right\Vert _{\mathcal{F}L^{s,p}}\leq\left\Vert u(\cdot,0)\right\Vert _{\mathcal{F}L^{s,p}}$,
$u_{N}$ is defined on the interval where $u$ is defined, and $u_{N}\rightarrow u$
in $C([0,T],\mathcal{F}L^{s,p})$. Since $\widehat{u}_{N}(\cdot,0)$
is compactly supported, $u_{N}\in C([0,T_{0}],\mathcal{F}L^{\sigma,p})\subset C([0,T_{0}],\mathcal{F}L^{1})$
for some large $\sigma$. Here, $T_{0}$ depends on $\sigma$ and
$N$. Thus, if $t\leq T_{0}$, in (\ref{eq: a sum integral}) we can
exchange the order of the sum and the integral, therefore $u_{N}$
satisfies (\ref{eq: a integral sum}). Thus, $u_{N}$ is a classical
solution of (\ref{eq:modified mKdV}). Using the bound (\ref{eq:converge}),
we can repeat the argument on the interval $[T_{0},2T_{0}]$, etc.,
and show that $u_{N}$ is a classical solution on an interval $[0,T_{1}]$
where $T_{1}$ depends on $\left\Vert u_{0}\right\Vert _{\mathcal{F}L^{s,p}}$
only. Thus $u$ is the limit in $C([0,T_{1}],\mathcal{F}L^{s,p})$
of smooth solutions $u_{N}$.
\end{proof}
The proof of Proposition 1.2 is basically the same as that of Proposition
1.4 in \cite{MR2333210}, hence we obmit it.\\

\bibliographystyle{amsplain}
\bibliography{KdV}

\noun{\footnotesize ~}{\footnotesize \par}

\noun{\footnotesize Department of Mathematics, University of Chicago,
5734 S. University Ave., Chicago, IL 60637, USA}{\footnotesize \par}

\textit{\footnotesize E-mail address}: \texttt{\footnotesize tu@math.uchicago.edu}
\end{document}